\documentclass[12pt]{amsart}
  
\usepackage{amsmath}
\usepackage{amsthm}
\usepackage{amssymb}
\usepackage{mathrsfs}

\usepackage{booktabs}
\usepackage{siunitx}
\usepackage{multirow}
\usepackage{cite}

\usepackage[left=3cm,right=3cm, bottom=3cm]{geometry}
\usepackage{array}
\usepackage{enumerate}
\usepackage{color}
\usepackage[colorlinks=true, linkcolor=blue, citecolor=blue]{hyperref}
\numberwithin{equation}{section}
\usepackage{enumitem}
\setlist[enumerate,1]{label={\upshape(\roman*)}}

\theoremstyle{theorem}
\newtheorem{theorem}{Theorem}[section]
\newtheorem{corollary}[theorem]{Corollary}
\newtheorem{proposition}[theorem]{Proposition}
\newtheorem{lemma}[theorem]{Lemma}

\theoremstyle{definition}
\newtheorem{definition}[theorem]{Definition}
\newtheorem{remark}[theorem]{Remark}

\theoremstyle{remark}
\theoremstyle{proof}

\newcommand{\N}{\mathbb{N}}

\newcommand{\R}{\mathbb{R}}

\renewcommand{\P}{\mathbb{P}}

\newcommand{\bE}{\mathbb{E}}

\newcommand{\cA}{\mathcal{A}}

\newcommand{\cF}{\mathcal{F}}
\newcommand{\cG}{\mathcal{G}}
\newcommand{\cP}{\mathcal{P}}
\newcommand{\cT}{\mathcal{T}}

\newcommand{\I}{\mathrm{I}}
\newcommand{\E}{\mathrm{E}}
\renewcommand{\H}{\mathrm{H}}

\newcommand{\wt}{\widetilde}

\newcommand{\varep}{\varepsilon}

\newcommand{\norm}[1]{\| #1 \|}

\newcommand{\Paren}[1]{\left( #1 \right)}

\newcommand{\Abs}[1]{\left| #1 \right|}

\newcommand{\dkol}[1]{d_{\mathrm{K}}(#1)}
\newcommand{\dhel}{d_{\mathrm{H}}}
\newcommand{\dpro}{d_{\mathrm{P}}}
\newcommand{\dlev}{d_{\mathrm{L}}}
\newcommand{\dtv}[1]{d_{\mathrm{TV}}(#1)}

\newcommand{\Tal}{\mathrm{Tal}}
\newcommand{\CE}{\mathrm{CE}}

\newcommand{\Var}{\operatorname{Var}}

\begin{document}

\title{Deficit estimates for the Logarithmic Sobolev inequality}

\author{Emanuel Indrei}
\address{Department of Mathematics, Purdue University, 150 N. University Street, West Lafayette, IN 47907-2067, USA}
\email{eindrei@purdue.edu}

\author{Daesung Kim}

\address{Department of Mathematics, Purdue University, 150 N. University Street, West Lafayette, IN 47907-2067, USA}
\email{daesungkim@purdue.edu}

\date{\today}
\begin{abstract}
We identify sharp spaces and prove quantitative and non-quantitative stability results for the logarithmic Sobolev inequality involving Wasserstein and $L^p$ metrics. The techniques are based on optimal transport theory and Fourier analysis. We also discuss a probabilistic approach.    	
\end{abstract}
\maketitle

\section{Introduction}
For a probability measure $fd\gamma$ absolutely continuous with respect to the Gaussian measure $d\gamma=(2\pi)^{-\frac{n}{2}}e^{-\frac{|x|^{2}}{2}}dx$ such that $\sqrt{f}\in W^{1,2}(\R^{n},d\gamma)$, the Fisher information and the relative entropy are given, respectively, by
\begin{eqnarray*}
	\I(f) &=& \int \frac{|\nabla f|^{2}}{f}d\gamma,\\
	\H(f) &=& \int f\log f d\gamma.
\end{eqnarray*}
The classical logarithmic Sobolev inequality LSI states
\begin{eqnarray}\label{eq:LSI}
	\delta(f)=\frac{1}{2}\I(f)-\H(f)\ge 0,
\end{eqnarray}
where $\delta$ is the LSI deficit. There are several proofs based on the central limit theorem \cite{Gross1975}, Ornstein--Uhlenbeck semigroup \cite{MR1160084}, Prekopka--Leindler inequality \cite{MR1800062}, optimal transport theory \cite{MR1894593}, and harmonic analysis \cite{Carlen1991, MR1254832}. 
 
Carlen characterized the equality cases in two ways: first, if $f \in L^p(\mathbb{R}^{2n})$ is a product function in $(x,y)$ and $\Big(\frac{x+y}{\sqrt{2}}, \frac{x-y}{\sqrt{2}}\Big)$, then $f$ as well as its factors are all Gaussian functions; thereafter, he proved a Minkowski-type inequality and derived the strict superadditivity of the Fisher information which combined with the factorization theorem yields that equality holds in \eqref{eq:LSI} only if $e^{b\cdot x-\frac{b^{2}}{2}},$ $b \in \mathbb{R}^n$. 

The second proof is based on the Beckner--Hirschman entropic uncertainty principle:  
let    
$$
U: L^2(\mathbb{R}^n, dx) \rightarrow L^2(\mathbb{R}^n, dm)
$$
be defined by $Uh(x)=2^{-n/4}e^{\pi|x|^2}h(x)$ and denote the Fourier transform by 
$$\hat f (\xi)=\mathcal{F}f(\xi)=\int e^{-2\pi i\xi \cdot x}f(x)dx.$$
The Fourier-Wiener transform is $\mathcal{W}=U \mathcal{F} U^*$, where $U^*$ is the adjoint of $U$. If $dm=2^{\frac{n}{2}} e^{-2\pi |x|^2}$ and $f \in L^2(dm)$ is normalized, Carlen showed that the entropic uncertainty principle is equivalent to
\begin{equation} \label{carlen}
\int |\mathcal{W}f|^2\log |\mathcal{W}f|^2 dm \le \frac{1}{2\pi} \int |\nabla f|^2 dm -\int |f|^2\log|f|^2 dm=:\delta_c(f),
\end{equation}
and if $f \ge 0$, 
$$
\int |\mathcal{W}f|^2\log |\mathcal{W}f|^2 dm=0
$$
only if
$$
f_a(x)=e^{2\pi(a \cdot x-\frac{|a|^2}{2})},
$$  
where $a \in \mathbb{R}^n$. 
      
Note that a rescaling of \eqref{carlen} improves \eqref{eq:LSI}. In this paper, we are interested in measuring the deviation of a function from the class of optimizers. Let $\cA$ be a family of probability measures and $d$ a metric/functional that identifies the equality cases. We say that LSI is \emph{weakly stable} under $(d,\cA)$ if $\delta(f)\to 0$ ($\delta_c(f)\rightarrow 0)$ implies $d(fd\gamma,d\gamma)\to 0$ ($d(|f|^2dm,dm)\rightarrow 0$) for centered measures and \emph{stable} if a modulus of continuity is explicit. Note that the Gaussian measure is the only centered optimizer.  

There has been much interest devoted to finding stability bounds \cite{Indrei2014, Bobkov2014, MR3758731, Dolbeault2016, Feo2017, MR3665794, arx} (see also \cite{MR3320893, MR3682667}). In particular, the first results involving a metric appeared in \cite{Indrei2014, Bobkov2014} via the quadratic Wasserstein distance
$$
\delta(f) \ge c W_2^2(d\nu_b, d\gamma)
$$
for a class of probability measures $\nu_b$ satisfying certain differential constraints. $W_2$--stability cannot hold for all probability measures since it would improve the sharp LSI constant \cite[Remark 4.3]{Indrei2014}. It was shown in \cite{Fathi2016} that probability measures for which a $(2,2)$-Poincar\'e inequality holds satisfy the following improvement of LSI 
$$
H(f) \le \frac{c(\lambda)}{2}I(f)\\
$$
\noindent where $\lambda$ is the Poincar\'e constant and $c(\lambda)=\frac{1-\lambda+\lambda\log \lambda}{(1-\lambda)^2}<1$; in particular, this yields $W_2$--stability for this function space and 
$$
\int |f-1| d\gamma \le \tilde c_\lambda \delta^{\frac{1}{2}}(f).
$$

The most general class for which $L^1$--stability holds is still not fully understood. To this aim, let 
$$\cP_{2}^{M}(\R^{n})=\{\mu: m_{2}(\mu)\leq M\},$$  
where $M\ge n$ and $m_{2}(\mu)$ is the second moment of the probability measure $\mu$. The following holds.

\begin{theorem}\label{mainthm3}
	Let $f d\gamma$ be a centered probability measure in $\cP_{2}^{M}(\R)$ ($M\ge1$). Then there exists $C_M>0$ such that 
	$$\|f-1\|_{L^1(d\gamma)} \le C_M\delta^{\frac{1}{4}}(f).$$ 
\end{theorem}

This result is false if $p>1$ \cite[Theorem 1.1]{Kim2018}: there exists a sequence of centered probability measures $f_kd\gamma \in \mathcal{P}_2^M(\mathbb{R})$ for which $\delta(f_k) \rightarrow 0$ and 
$$
\displaystyle \liminf_{k \rightarrow \infty} \|f_k-1\|_{L^p(d\gamma)} >0.
$$
A sufficient additional condition for $L^p$--stability is higher integrability.

\begin{corollary}\label{maincor3}
	Let $f d\gamma$ be a centered probability measure in $\cP_{2}^{M}(\R)$ ($M\ge1$) such that $f \in \{f:\int |f|^{2p-1}d\gamma \le N\}$, $p>1$. Then there exists $C=C(M,N,p)>0$ such that 
	$$\|f-1\|_{L^p(d\gamma)} \le C\delta^{\frac{1}{8p}}(f).$$ 
\end{corollary}

Note that $\cP_{2}^{M}(\R^{n})$ contains the space of probability measures satisfying a $(2,2)$-Poincar\'e inequality with parameter $\lambda>0$ for some $M=M(\lambda)>0$. One may tensorize this inequality and also consider the case of product measures.  

\begin{corollary} \label{cor3}
Let $fd\gamma$ be a probability measure such that $$f \in \mathcal{F}_M=\Big\{f: \int |x_i|^2f(x_1,\ldots,x_{i-1},x_i,x_{i+1},\ldots,x_n)d\gamma(x_i) \le M\Big\}$$ ($M \ge 1$) and $x_i \mapsto f(x_1,\ldots,x_{i-1},x_i,x_{i+1},\ldots,x_n)$ is centered.
Then there exists $C=C(n,M)>0$ such that  
$$\int |f-1| d\gamma \le C \delta^{\frac{1}{4}}(f).$$ 
\end{corollary}

\begin{corollary} \label{cor4}
Suppose $f(x_1,\ldots,x_n)=\Pi_{i=1}^n f_i(x_i)$, where $f_i \in \cP_{2}^{M}(\R)$ ($M \ge 1$) and $fd\gamma$ is a centered probability measure.   
Then there exists $C=C(n,M)>0$ such that 
$$\int |f-1| d\gamma \le C \delta^{\frac{1}{4}}(f).$$ 
\end{corollary}

The method of proof is based on two lower bounds on the deficit which depend on the solution to the optimal transportation problem from which one may deduce that the total variation is bounded above in terms of $W_1$ and the deficit and then employing the following $W_1$--stability result.

\begin{theorem}\label{mainthm1}
	Let $fd\gamma$ be a centered probability measure in $\cP_{2}^{M}(\R^{n})$ ($M \ge n$). There exists $C=C(n,M)>0$ such that 
	\begin{eqnarray*}
		\delta(f)
		\geq		C\min\{W_{1}(fd\gamma, d\gamma),W_{1}^{4}(fd\gamma, d\gamma)\}.
	\end{eqnarray*}
\end{theorem}

$W_1$--stability is not true in $$\cP_{1}(\R^{n})=\bigcup_{M>0} \cP_{1}^{M}(\R^{n}),$$ i.e. the constant $C(n,M)$ cannot be taken independent of $M$ \cite[Theorem 1.2]{Kim2018}. 

A scaling argument shows (see e.g. \cite{Dolbeault2016, Bobkov2014})
\begin{eqnarray*}
	2\delta(f)
	\geq \Big(2\H(f)+(m_{2}(\gamma)-m_{2}(\nu))\Big)^{2},
\end{eqnarray*}
cf. Theorem \ref{aat} and \cite[(33)]{MR3758731}.
In particular, if $m_{2}(\gamma)=n\geq m_{2}(\nu)$ and $d\nu=fd\gamma$, then Talagrand's inequality \eqref{ineq:Talagrand} implies 
\begin{eqnarray*}
	2\delta(f)
	\geq W_{2}^{4}(d\nu,d\gamma)+(m_{2}(\gamma)-m_{2}(\nu))^{2}.
\end{eqnarray*}
However, for every $\epsilon>0$, there exists $d\nu_k=f_kd\gamma \in\cP_{2}^{n+\epsilon}(\R^{n})$ such that $\delta(f_k) \rightarrow 0$ as $k \rightarrow \infty$ and
$$ 
\liminf_{k \rightarrow \infty} W_{2}(d\nu_k,d\gamma)>0
$$  
\cite[Theorem 1.1]{Kim2018}.

The proof of Theorem \ref{mainthm1} is based on the observation that the entropy is bounded by the deficit and the second moment via the HWI inequality which combines with a stability bound for Talagrand's inequality \cite{Cordero-Erausquin2017, Fathi2016}. 

The following theorem (also based on an optimal transport approach) does not impose additional regularity assumptions or second moment bounds and yields $L^1$--stability in case that there is an $L^1$ bound on the densities. 

\begin{theorem} \label{talpha}
Let $fd\gamma$ be a centered probability measure such that $f \in \{ f \ge \alpha\}$ for some $\alpha \in (0,1]$. Then 
$$ 
\|\log f-L\|_{L^1(d\gamma)} \le C(\alpha,n) \delta^{\frac{1}{2}}(f),
$$
where $L(x)=a_f\cdot x+b_f$, and $C(\alpha,n)>0.$ Moreover, $|a_f| +|b_f|\le c$ for $c=c(n,\alpha)>0$. 
\end{theorem}

\begin{corollary} \label{c1}
Let $\{f_kd\gamma\}$ be a sequence of centered probability measures such that $f_k \in \{ f \ge \alpha\}$ for some $\alpha \in (0,1]$ and $\delta(f_k) \rightarrow 0$ as $k \rightarrow \infty$. Then $f_{k_j} \rightarrow c$ a.e. as $j \rightarrow \infty$, for some subsequence $\{f_{k_j}\} \subset \{f_k\}$ and constant $c\in [\alpha,1]$. 
\end{corollary} 

\begin{corollary} \label{c2}
Let $\{f_kd\gamma\}$ be a sequence of centered probability measures such that $f_k \in \{ f \ge \alpha\}$ for some $\alpha \in (0,1]$ and $f_k \le g$ for $g \in L^1(d\gamma)$. If $\delta(f_k) \rightarrow 0$ as $k \rightarrow \infty$, then $f_k \rightarrow 1$ in $L^1(d\gamma)$. 
\end{corollary}

Weak stability results can also be proved by investigating conditions for which

$$
\int |\mathcal{W}f|^2\log |\mathcal{W}f|^2 dm \rightarrow 0
$$ 
implies
$$
f \rightarrow 1.
$$
This turns out to be the case if the density's growth is controlled. 

\begin{theorem} \label{c5}
Let $\{f_k\}$ be normalized and centered in $L^2(dm)$ and suppose $f_k \in \{f:\int |f|^2e^{-(2\pi-\epsilon) |x|^2}dx \le M\}$, where $\epsilon>0$. Then if $\delta_c(f_k) \rightarrow 0$ as $k\rightarrow \infty$, 
$$
f_k \rightarrow 1
$$ 
in $L^2(dm)$.
\end{theorem}

A change of variables yields the following.

\begin{corollary} \label{corc5}
Let $\{f_kd\gamma\}$ be centered probability measures and suppose $f_k \in \{f:\int fe^{-(\frac{1}{2}-\frac{\epsilon}{4\pi}) |x|^2}dx \le C\}$, where $\epsilon>0$ and $C>0$. Then if $\delta(f_k) \rightarrow 0$ as $k\rightarrow \infty$, 
$$
f_k \rightarrow 1
$$ 
in $L^1(d\gamma)$.
\end{corollary}

\begin{corollary} \label{corc6}
Let $\{f_kd\gamma\}$ be centered probability measures and suppose $f_k \in \{f:\int fe^{-(\frac{1}{2}-\frac{\epsilon}{4\pi}) |x|^2}dx \le C\} \cap \{f: \int|f|^{2p-1}d\gamma \le N\}$, where $\epsilon>0$ and $C>0$. Then if $\delta(f_k) \rightarrow 0$ as $k\rightarrow \infty$, 
$$
f_k \rightarrow 1
$$ 
in $L^p(d\gamma)$,
$p \ge 1$.
\end{corollary}

\begin{remark}
If $p>1$, then for $\epsilon<\frac{2\pi(2(p-1))}{2p-1}$, the constant 
$$C=(2\pi)^\frac{n}{2}N^{\frac{1}{2p-1}}\bigg(\frac{1}{1-\frac{\epsilon(2p-1)}{4\pi(p-1)}}\bigg)^{\frac{n(p-1)}{2p-1}}$$
may be selected.
\end{remark}

Combining the optimal transport method and \eqref{carlen} yields the following inequality, which in particular, implies weak $L^2$--stability for $\cP_{2}^{M}(\R^{n})$ with respect to $\delta_c$ and weak $L^1$--stability with respect to $\delta$.

\begin{theorem} \label{aat} 
Let $f$ be normalized in $L^2(dm)$. Then
$$
2\pi \int |x|^2dm-2\pi \int |x|^2|f|^2dm +\frac{1}{2\pi} \int |\nabla f|^2 dm \le 2\sqrt{\pi n} \delta_c(f)^{\frac{1}{2}}+\delta_c(f).
$$
\end{theorem}

\begin{corollary} \label{c3}
Let $\{f_k\}$ be normalized and centered in $L^2(dm)$ and suppose $f_k\in \{f:\int |f|^4dm \le M\}$ ($M\ge 1$). Then if $\delta_c(f_k) \rightarrow 0$ as $k\rightarrow \infty$, $$\int |\nabla f_k|^2 dm \rightarrow 0$$ 
and in particular,
$$f_k \rightarrow 1$$ in $L^2(dm)$.  
\end{corollary}

\begin{corollary} \label{c4}
Let $\{f_k\}$ be normalized and centered in $L^2(dm)$ and suppose $f_k\in \{f:\int |x|^2|f|^2dm \le M\}$ ($M\ge \frac{n}{4\pi}$). Then if $\delta_c(f_k) \rightarrow 0$ as $k\rightarrow \infty$, $f_k \rightarrow 1$ in $L^2(dm)$.  
\end{corollary}

The rescaled version of Corollary \ref{c4} can also be shown with a proof based solely on optimal transport methods.

\begin{theorem}\label{mainthm2}
	Let $\{f_{k}d\gamma\}$ be a sequence of centered probability measures in $\cP_{2}^{M}(\R^{n})$ ($M \ge n$). If $\delta(f_k)\to 0$ as $k\to\infty$, then $f_{k}\to 1$ in $L^{1}(\R^{n},d\gamma)$.  
\end{theorem}

Another approach to proving stability estimates for LSI is to investigate quantitative versions of Cram\'er's theorem \cite{Bobkov2016, MR2918082},  which states that if the convolution of two non-negative integrable functions is Gaussian, then the functions must be Gaussian\footnote{This fact was already utilized in Carlen's proof of the equality cases.} and combine them with a convolution type deficit bound \cite{Feo2017}. It turns out that this technique yields several results which we include in the appendix.    

\begin{table}[t]
\small 
\begin{center}
\begin{tabular}{|c|c|c|c|}
\hline
{Metric} & {Space} & {Stability} & {Reference} \\ \hline\hline

\multirow{2}{*}{$W_{1}$} 
& {$\cP_{2}^{M}$ ($M\geq n$)}	& {stable} & {\parbox{3cm}{Theorem \ref{mainthm1}; cf. \cite[Corollary 6]{Fathi2016}}} \\\cline{2-4}
& {$\cP_{2}$} & {unstable} & {\cite[Theorem 1.2]{Kim2018}} \\\hline

\multirow{4}{*}{$W_{2}$} 
& $\cP_{2}^{n}$ & {stable} & {\cite[Corollary 1.2]{Bobkov2014}} \\\cline{2-4}
& {(2,2)-Poincar\'e} & {stable} & {\cite[Corollary 3]{Fathi2016} }	 \\\cline{2-4}
& {$\{f:(-1+\varep)\leq D^{2}(\log\frac{1}{f})\leq M\}$, $\varep>0$} & {stable} & {\cite[Theorem 1.1]{Indrei2014} }	 \\\cline{2-4}
& {$\cP_{2}^{M}$ ($M>n$)} & {unstable} & {\cite[Theorem 1.1]{Kim2018}} \\\hline

\multirow{6}{*}{$L^{1}$} 
& $\cP_{2}^{M}$  ($M\geq n=1$)& {stable} & {Theorem \ref{mainthm3}}	 \\\cline{2-4}
& $\cP_{2}^{n}$ & {stable} & {\cite[Corollary 1.2]{Bobkov2014}} \\\cline{2-4} 
& $\cP_{2}^{M}$  ($M\geq n$)& {weakly stable} & {\parbox{2.5cm}{Theorem \ref{mainthm2}; Corollary \ref{c4}}} \\\cline{2-4}
& {(2,2)-Poincar\'e} & {stable} & {\cite[Corollary 4]{Fathi2016} }	 \\\cline{2-4}
& $\{f:\int fe^{-(\frac{1}{2}-\frac{\epsilon}{4\pi}) |x|^2}dx \le C\}$, $\epsilon>0$ & {weakly stable} & {\parbox{2.5cm}{Theorem \ref{c5}; Corollary \ref{corc5}}} \\\cline{2-4}
& $\{f: \alpha\le f \le g\}$, $\alpha \in (0,1]$, $g \in L^1(d\gamma)$ & {weakly stable} & {Corollary \ref{c2}}  \\\hline
\multirow{3}{*}{$L^{p}$ ($p>1$)} 
& {$\cP_{2}^{M}$ ($M>n$)} & {unstable}	 & {\cite[Theorem 1.1]{Kim2018}}	 \\\cline{2-4}
& {$\cP_{2}^{M}$ ($M\ge n=1$)}, $\{\int|f|^{2p-1}d\gamma \le N\}$ & {stable}	 & {Corollary \ref{maincor3}}	 \\\cline{2-4}
& {$\cP_{2}^{M}$ ($M\ge n$)}, $\{\int|f|^{2p-1}d\gamma \le N\}$ & {weakly stable}	 & {Corollary \ref{corc6}}	 \\\hline
\end{tabular}	
\end{center}
\caption{A summary of stability/instability results for LSI.}
\label{table}
\end{table}

\section{Preliminaries}
\subsection{The Wasserstein distances}
For $p\geq 1$ and a probability measure $\mu$, the $p$-th moment of $\mu$ is given by $m_{p}(\mu)=\int_{\R^{n}}|x|^{p}d\mu$. The space of probability measures with finite $p$-th moment is denoted by $\cP_{p}(\R^{n})$. The \emph{Wasserstein distance of order $p$} between two probability measures $\mu,\nu\in \cP_{p}(\R^{n})$ is 
\begin{eqnarray*}
	W_{p}(\mu,\nu)=\inf_{\pi} \Big(\iint_{\R^{n}\times \R^{n}} |x-y|^{p}d\pi(x,y)\Big)^{\frac{1}{p}}
\end{eqnarray*}
where the infimum is taken over all probability measures $\pi$ on $\R^{n}\times \R^{n}$ with marginals $\mu$ and $\nu$. In general, one can define the optimal transportation cost with a cost function $c(x,y)$ on $\R^{n}\times \R^{n}$ by
\begin{eqnarray*}
	\cT_{c}(\mu,\nu)=\inf_{\pi} \Big(\iint_{\R^{n}\times \R^{n}} c(x,y)d\pi(x,y)\Big).
\end{eqnarray*}
In particular, $W_{1}$ is called the Kantorovich--Rubinstein distance and $W_{2}$ is called the quadratic Wasserstein distance. 

We recall some properties. For $p\geq 1$, $W_{p}$ defines a metric on $\cP_{p}(\R^{n})$. If $p_{1}<p_{2}$, then $\cP_{p_{1}}(\R^{n})$ contains $\cP_{p_{2}}(\R^{n})$ and the $W_{p_{1}}$ distance is weaker than $W_{p_{2}}$. Indeed it follows from Jensen's inequality that $W_{p_{1}}(\mu,\nu)\leq W_{p_{2}}(\mu,\nu)$ for probability measures $\mu,\nu$ in $\cP_{p_{2}}(\R^{n})$. The Wasserstein distance of order $p$ is stronger than weak convergence: let $\nu_{k}$ be a sequence of probability measures in $\cP_{p}(\R^{n})$, then $\nu_{k}$ converges to $\mu$ in $W_{p}$ if and only if $\nu_{k}\rightharpoonup \mu$ weakly and $m_{p}(\nu_{k})\to m_{p}(\mu)$ as $k \rightarrow \infty$.

Let $\mu$ and $\nu$ be probability measures with finite second moments. Then there exists a map $T:\R^{n}\to \R^{n}$ such that $\nu(A)=\mu(T^{-1}(A))$ for all Borel sets $A$ in $\R^{n}$ and
\begin{eqnarray*}
	W_{2}^{2}(\mu,\nu)                                                        
	&=& \int_{\R^{n}}|T(x)-x|^{2}d\mu.
\end{eqnarray*}
It is well-known that the map $T$ is uniquely determined $\mu$-almost everywhere and is a gradient of a convex function called the \emph{Brenier map}.

We say a function $\varphi$ is 1-Lipschitz if $|\varphi(x)-\varphi(y)|\leq |x-y|$ for all $x,y\in\R^{n}$. The Kantorovich--Rubinstein distance $W_{1}$ has a dual form
\begin{eqnarray*}
	W_{1}(\mu,\nu)=\sup\Big\{\int_{\R^{n}}\varphi(d\mu-d\nu):\varphi\in L^{1}(d|\mu-\nu|), \varphi\text{ is 1-Lipschitz.}\Big\}.
\end{eqnarray*}
For further information, we refer the reader to \cite{Villani2003}.

\subsection{The total variation distance}
Let $\mu$ and $\nu$ be probability measures. The total variation distance between $\mu$ and $\nu$ is defined by
\begin{eqnarray*}
	\dtv{\mu,\nu}=\sup_{A}|\mu(A)-\nu(A)|
\end{eqnarray*}
where the supremum is taken over all Borel sets in $\R^{n}$ and yields a stronger topology than weak convergence. That is, if $\dtv{\mu,\nu_{k}}\to 0$ as $k\to\infty$, then $\nu_{k}$ converges weakly to $\mu$ (however, the converse does not hold). The total variation distance can be thought of as the optimal transportation cost with cost function $c(x,y)=1_{x\neq y}$ and has a dual form 
\begin{eqnarray*}
	\dtv{\mu,\nu}=\sup_{ 0\leq |\varphi|\leq 1}\int_{\R^{n}}\varphi(d\mu-d\nu).
\end{eqnarray*}
Suppose that $\nu$ is absolutely continuous with respect to $\mu$ and has $f$ as its Radon--Nikodym derivative, that is, $d\nu=fd\mu$. Then the total variation distance $\dtv{\mu,\nu}$ can be written in terms of the $L^{1}$--norm
\begin{eqnarray*}
	\dtv{\mu,\nu}=\frac{1}{2}\norm{f-1}_{L^{1}(d\mu)}.
\end{eqnarray*}
It is well-known that the total variation distance is comparable to the Hellinger distance
\begin{eqnarray}\label{ineq:L1_L2}
	\norm{\sqrt{f}-1}_{L^{2}(d\mu)}^{2}
	\leq \norm{f-1}_{L^{1}(d\mu)}
	\leq 2\norm{\sqrt{f}-1}_{L^{2}(d\mu)}.
\end{eqnarray}

\subsection{Some inequalities}
Let $\gamma$ be the Gaussian measure and $d\nu=fd\gamma$. The entropy functional $\nu \mapsto \H(f)$ is utilized to measure how far $\nu$ is from the Gaussian measure $\gamma$. It is stronger than the total variation distance but weaker than the $L^{p}$-norm for $p>1$ 
\begin{eqnarray*}
	2\norm{f-1}_{L^{1}(d\gamma)}^{2}
	\leq \H(f)
	\leq \frac{2}{p-1}\norm{f-1}_{L^{p}(d\gamma)}^{p}+2\norm{f-1}_{L^{p}(d\gamma)}.
\end{eqnarray*}
The first inequality is called Pinsker's inequality. Talagrand \cite{Talagrand1996} introduced the inequality
\begin{eqnarray}\label{ineq:Talagrand}
	W_{2}^{2}(d\nu,d\gamma)\leq 2\H(f),
\end{eqnarray}
which implies that the entropy is stronger than the quadratic Wasserstein distance. Otto and Villani \cite{Otto2000} proved that LSI implies the Talagrand inequality. Let $\delta_{\Tal}(f):=2\H(f)-W_{2}^{2}(d\gamma,d\nu)$. If $\nu$ in $\cP_{2}$ is centered, then Cordero-Erausquin \cite{Cordero-Erausquin2017} showed 
	\begin{eqnarray*}
		\delta_{\Tal}(f) 
		\geq C \min\Paren{W_{1}^{2}(d\gamma,d\nu),W_{1}(d\gamma,d\nu)}
	\end{eqnarray*}	
	where $C>0$, and a comparable stability result was also shown in \cite{Fathi2016}. The quantitative Talagrand inequality is one of the main ingredients for proving Theorem \ref{mainthm1}.
Otto and Villani proved the HWI inequality which is an ``interpolation'' inequality between the entropy, Wasserstein distance, and Fisher information 
\begin{eqnarray*}\label{ineq:HWI}
	\H(f)\leq W_{2}(d\nu,d\gamma)\sqrt{\I(f)}-\frac{1}{2}W_{2}^{2}(d\nu,d\gamma).
\end{eqnarray*}

\section{Proofs of The Main Results}

\begin{proof}[Proof of Theorem \ref{mainthm3}]
Let $T$ be the Brenier map between $fd\gamma$ and $d\gamma$. Then it follows that 
$$\int |T(x)-x+\nabla \log f|^2 fd\gamma \le 2 \delta(f).$$ Next by Poincare, 
\begin{align*}
\int |f-1| d\gamma &\le \int |\nabla f| d\gamma=\int |\nabla \log f| fd\gamma\\
&\le \int |\nabla \log f-T(x)+x| fd\gamma + \int |T(x)-x| fd\gamma\\
&\le \big(\int |\nabla \log f-T(x)+x|^2 fd\gamma \big)^{\frac{1}{2}}+\int |T(x)-x| fd\gamma\\
&\le \sqrt{2}\delta^{\frac{1}{2}}(f)+\int |T(x)-x| fd\gamma.  
\end{align*}
Therefore,  
$$\int |f-1| d\gamma-\sqrt{2}\delta^{\frac{1}{2}}(f) \le \int |T(x)-x| fd\gamma.$$ 
Since $T$ pushes $fd\gamma$ onto $d\gamma$, the previous inequality implies
$$\int |f-1| d\gamma-\sqrt{2}\delta^{\frac{1}{2}}(f) \le W_1(fd\gamma, d\gamma)$$ 
and Theorem \ref{mainthm1} yields
$$\int |f-1| d\gamma\le \sqrt{2}\delta^{\frac{1}{2}}(f) + C\max(\delta(f), \delta^{\frac{1}{4}}(f)).$$
\end{proof}

\begin{proof}[Proof of Corollary \ref{maincor3}]
\begin{align*}
\int |f-1|^p d\gamma&=\int |f-1|^{p-\frac{1}{2}}|f-1|^\frac{1}{2}d\gamma\\
&\le (2^{2p-2}(N+1))^\frac{1}{2}||f-1||_{L^1(d\gamma)}^\frac{1}{2}\\
&\le C_M(2^{2p-2}(N+1))^\frac{1}{2} \delta^{\frac{1}{8}}(f).
\end{align*}

\end{proof}

\begin{proof}[Proof of Corollary \ref{cor3}] 
For fixed $x'=(x_1,\ldots,x_{i-1},x_{i+1},x_n)$, let 
$$
g_{x'}(x_i)=f(x_1,\ldots,x_{i-1},x_i,x_{i+1},\ldots,x_n).
$$
Theorem \ref{mainthm3} implies 
\begin{align*}
\int g_{x'}(x_i) \log g_{x'}(x_i) d\gamma(x_i)+c\Big(\int |g_{x'}(x_i)-1| d\gamma(x_i) \Big)^4 & \le
\frac{1}{2} \int \frac{(\partial_{x_i} g_{x'}(x_i))^2}{g_{x'}(x_i)}d\gamma(x_i).
\end{align*}
In particular
\begin{align*}
\int \int g_{x'}(x_i) \log g_{x'}(x_i) d\gamma(x_i)d\gamma(x')+\int \Big(\int |g_{x'}(x_i)-1| &d\gamma(x_i) \Big)^4d\gamma(x')\\ 
&\le \frac{1}{2} \int \int \frac{(\partial_{x_i} g_{x'}(x_i))^2}{g_{x'}(x_i)}d\gamma(x_i)d\gamma(x')\\
&\le \frac{1}{2} \int \int \frac{|\nabla f|^2}{f}d\gamma(x_i)d\gamma(x').
\end{align*}
Therefore, 
$$
\int \Big(\int |g_{x'}(x_i)-1| d\gamma(x_i) \Big)^4d\gamma(x') \le c\delta(f) 
$$
and this implies 
$$
\int |f-1| d\gamma \le c\delta^{\frac{1}{4}}(f). 
$$
\end{proof}

\begin{proof}[Proof of Corollary \ref{cor4}]
By applying Theorem \ref{mainthm3} to $f_i$, it follows that 
$$ 
\sum_{i=1}^n \Big(\int |f_i(x_i)-1| d\gamma(x_i) \Big)^4 \le c \sum_{i=1}^n \delta(f_i).
$$
Since 
$$
|\nabla f_1f_2\cdots f_n|^2 =\sum_{i=1}^n (\partial_{x_i} f_i(x_i))^2\Pi_{j \neq i}(f_j(x_j))^2,
$$

$$
\int \frac{|\nabla f_1f_2\cdots f_n|^2}{f}d\gamma=\sum_{i=1}^n \int \frac{(\partial_{x_i} f_i)^2}{f_i}d\gamma(x_i).
$$
Moreover, 
$$
\int f\log f d\gamma=\sum_{i=1}^n \int f_i(x_i)\log f_i(x_i)d\gamma(x_i)
$$
and 
$$
\delta(f)=\sum_{i=1}^n \delta(f_i);
$$
thus, 
$$ 
\sum_{i=1}^n \int |f_i(x_i)-1| d\gamma(x_i) \le c \delta^{\frac{1}{4}}(f)
$$
and since
$$
\int |f-1|d\gamma \le \int |f_1-1|d\gamma+\int |f_2-1|d\gamma+\cdots+\int |f_n-1|d\gamma,
$$
the result follows.
\end{proof}

\begin{proof}[Proof of Theorem \ref{mainthm1}]

The proof is based on the stability of Talagrand inequality in \cite{Cordero-Erausquin2017, Fathi2016}. We denote by $\delta_{\Tal}(f):=2\H(f)-W_{2}^{2}(fd\gamma,d\gamma)$. D. Cordero-Erausquin shows that if $fd\gamma\in\cP_{2}$ is a centered probability measure, then
	\begin{eqnarray}\label{Talstab_Cordero}
		\delta_{\Tal}(f)
		\geq C_{\CE}\min\Paren{W_{1}^{2}(fd\gamma,d\gamma),W_{1}(fd\gamma,d\gamma)}
	\end{eqnarray}	
	where $C_{\CE}$ is a universal constant.
	First, we use the HWI inequality and Young's inequality to obtain
	\begin{eqnarray*}
		\H(f)
		&\leq& W_{2}(fd\gamma,d\gamma)\sqrt{\I(f)}-\frac{1}{2}W_{2}^{2}(fd\gamma,d\gamma) \\
		&\leq& \frac{1}{2t}\I(f)+\frac{t-1}{2}W_{2}^{2}(fd\gamma,d\gamma)
	\end{eqnarray*}
	for any $t>1$. 
	Suppose that a map $T:\R^{n}\to\R^{n}$ pushes forward $d\gamma$ to $fd\gamma$, then
	\begin{eqnarray*}
		W_{2}^{2}(fd\gamma,d\gamma)
		&\leq& \int_{\R^{n}}|T(x)-x|^{2}d\gamma \\
		&\leq& 2\Big(\int_{\R^{n}}|x|^{2}fd\gamma+\int_{\R^{n}}|x|^{2}d\gamma\Big) \\
		&\leq& 2(n+M).
	\end{eqnarray*}
	Since $\H(f)=\frac{1}{2}\I(f)-\delta(f)$ and $\H(f)\leq\frac{1}{2}\I(f)$, we have
	\begin{eqnarray}\label{thm1_eqn1}
		\Big(1-\frac{1}{t}\Big)\H(f)
		\leq (t-1)(n+M)+\delta(f).
	\end{eqnarray}
	By the HWI inequality \eqref{ineq:HWI}, we get 
	\begin{eqnarray*}
		\H(f)
		\leq -\frac{1}{2}(\sqrt{\I(f)}-W_{2}(fd\gamma,d\gamma))^{2}+\frac{1}{2}\I(f).
	\end{eqnarray*}
	Since $\sqrt{\I(f)}\geq \sqrt{2\H(f)}\geq W_{2}(fd\gamma,d\gamma)$, we have
	\begin{eqnarray}\label{thm1_eqn2}
		\delta(f)
		&\geq& \frac{1}{2}(\sqrt{2\H(f)}-W_{2}(fd\gamma,d\gamma))^{2} \nonumber\\
		&\geq& \frac{(2\H(f)-W_{2}^{2}(fd\gamma,d\gamma))^{2}}{2(\sqrt{2\H(f)}+W_{2}(fd\gamma,d\gamma))^{2}}\nonumber\\
		&\geq& \frac{\delta_{\Tal}^{2}(f)}{16\H(f)}.
	\end{eqnarray}
	In the last inequality, we have used Talagrand inequality. 
	Combining \eqref{thm1_eqn1} and \eqref{thm1_eqn2}, we obtain
	\begin{eqnarray*}
		\delta_{\Tal}^{2}(f)
		\leq \Big(\frac{16t}{t-1}\Big)\delta(f)(\delta(f)+(t-1)(n+M)).
	\end{eqnarray*}
	We solve this for $\delta(\nu)$ to see
	\begin{eqnarray*}
		\delta(f)
		\geq \frac{1}{2}(t-1)(n+M)\cF(K(t)\delta_{\Tal}(f))
	\end{eqnarray*}
	where $\cF(x)=\sqrt{x^{2}+1}-1$ and 
	\begin{eqnarray*}
		K(t):=\frac{1}{2\sqrt{t(t-1)}(n+M)}.
	\end{eqnarray*}
	Note that for $t>1$, the range of $K(t)$ is $(0,\infty)$. If we define $\cG(x):=\min\{x,x^{2}\}$, then one can easily see that $\cF(x)\geq (\sqrt{2}-1)\cG(x)$. It follows from \eqref{Talstab_Cordero} that
	\begin{eqnarray*}
		\delta(f)
		\geq \frac{\sqrt{2}-1}{2}(t-1)(n+M)\cG(C_{\CE}K(t)\cG(W_{1}(fd\gamma,d\gamma))).
	\end{eqnarray*}
	Since we have $C_{\CE}K(t)=1$ for
	\begin{eqnarray*}
		t=\frac{1}{2(n+M)}(\sqrt{(n+M)^{2}+C_{\CE}^{2}}+(n+M))>1,
	\end{eqnarray*}
	the proof is complete.
	\end{proof}

\begin{proof}[Proof of Theorem \ref{talpha}]
Let $T=\nabla \Phi=(T^1,T^2,\ldots,T^n)$ be the Brenier map from $fd\gamma$ to $d\gamma$ and $\{\lambda_i\}$ the eigenvalues of $DT-Id$. Then from the proof of LSI via optimal transport it follows that 

$$\int |T(x)-x+\nabla \log f|^2 fd\gamma \le 2 \delta(f);$$

$$ 
\int \sum_{i=1}^n (\lambda_i-\log(1+\lambda_i)) fd\gamma \le \delta(f).
$$

Without loss, $\delta(f) <<1$ (this follows from the argument below), and $C$ may change from line to line with dependence only on universal constants, $\alpha$, and $n$,   

\begin{align*}
\int \sum_{i=1}^n (\lambda_i-\log(1+\lambda_i)) fd\gamma &\ge C\sum_{i=1}^n\Big(\int_{\{|\lambda_i| \le c \}}|\lambda_i|^2 d\gamma+\int_{\{|\lambda_i| > c \}}|\lambda_i| d\gamma \Big)\\
&\ge C\sum_{i=1}^n\Big( \Big(\int_{\{|\lambda_i| \le c \}}|\lambda_i| d\gamma\Big)^2+\int_{\{|\lambda_i| > c \}}|\lambda_i| d\gamma \Big)\\
&\ge C\sum_{i=1}^n\Big( \Big(\int_{\{|\lambda_i| \le c \}}|\lambda_i| d\gamma\Big)^2+\Big (\int_{\{|\lambda_i| > c \}}|\lambda_i| d\gamma \Big)^2 \Big)\\
&\ge C\sum_{i=1}^n\Big(\int |\lambda_i| d\gamma \Big)^2\\
&\ge C\Big(\sum_{i=1}^n \int |\lambda_i| d\gamma \Big)^2\\
&\ge C\Big(\sum_{i,j} \int |\partial_jT^i-\delta_{ij}| d\gamma \Big)^2\\
&\ge C\Big(\sum_{i} \int |\nabla (T^i-x_i)| d\gamma \Big)^2\\
&\ge C\Big(\sum_{i} \int |T^i-x_i-a_i| d\gamma \Big)^2
\end{align*}
where $a=(a_1,\ldots,a_n)$ and 
$$
a=\int Td\gamma.
$$
Note that 
$$
|a| \le \int |T| d\gamma \le \frac{1}{\alpha} \int |x|d\gamma
$$
and
\begin{align*}
\int |\nabla \log f+a| d\gamma &\le \int |\nabla \log f+(T-x)| d\gamma+ \int |T-x-a| d\gamma\\
&\le \frac{1}{\sqrt{\alpha}}\Big(\int |\nabla \log f+(T-x)|^2 fd\gamma \Big)^{\frac{1}{2}}+\int |T-x-a| d\gamma\\
&\le C(\alpha,n)\delta^{\frac{1}{2}}(f),
\end{align*}
where $C(\alpha,n)=\sqrt{\frac{2}{\alpha}}+\frac{C(n)}{\alpha}$. Now let 
$$
b=\int \log f d\gamma \in (\log \alpha, 0);
$$
since 
$$
\int |\nabla \log f+a| d\gamma \ge c\int |\log f+a \cdot x-b| d\gamma,
$$
the result follows.
\end{proof}

\begin{proof}[Proof of Corollary \ref{c1}]
Theorem \ref{talpha} implies that 
$$ 
\int |\log f_ke^{-(a_{f_k}\cdot x + b_{f_k})}|d\gamma \rightarrow 0 
$$
as $k \rightarrow \infty$. There exists $a \in \mathbb{R}^n$ and $b \in \mathbb{R}$ such that $a_{f_k} \rightarrow a$ and $b_{f_k} \rightarrow b$ as $k \rightarrow \infty$, along a subsequence. Up to a further subsequence, it follows that 
$$
\log f_ke^{-(a_{f_k}\cdot x + b_{f_k})} \rightarrow 0
$$
a.e. as $k \rightarrow \infty$. In particular, 
$$
f_k \rightarrow e^{a\cdot x+b}
$$
a.e. as $k \rightarrow \infty$ and $\displaystyle \inf_{x,k} f_k(x) \ge \alpha$ implies $a=0$.      
\end{proof}

\begin{proof}[Proof of Corollary \ref{c2}]
Suppose 
$$
\limsup_{k \rightarrow \infty} \|f_k-1\|_{L^1(d\gamma)} >0.
$$
Corollary \ref{c1} implies $f_k \rightarrow c$ a.e. along a subsequence as $k \rightarrow \infty$ (taken from the subsequence achieving the $\limsup$) for some constant $c>0$ and the dominated convergence theorem implies $c=1$ (via the mass constraint), a contradiction.   
\end{proof}

\begin{proof}[Proof of Theorem \ref{c5}]
Suppose 
$$
\limsup_{k \rightarrow \infty} \|f_k-1\|_{L^2(dm)} > 0
$$  
and let $dm_\epsilon=e^{-(2\pi-\epsilon) |x|^2}dx$; 
since 
$$
\int |f_k|^2 dm_\epsilon \le C,
$$    
it follows that $f_k \rightharpoonup f$ weakly in $L^2(dm_\epsilon)$ along a subsequence. In particular, 
$$
\int e^{-2\pi i \xi \cdot x}U^*f_k(x)dx \rightarrow 0,
$$ 
as $k \rightarrow \infty$, and therefore, $\mathcal{W}f_k(\xi) \rightarrow \mathcal{W}f(\xi)$ for every $\xi \in \mathbb{R}^n$. However, since $\delta_c(f_k) \rightarrow 0$, it follows from Carlen's theorem that
$$
\int |\mathcal{W}f_k|^2 \log |\mathcal{W}_kf|^2 dm \le \delta_c(f_k) \rightarrow 0. 
$$ 
Therefore, 
$$
|\mathcal{W}f_k|^2 \rightarrow 1,  
$$
in $L^1(dm)$ as $k \rightarrow \infty$ and $|\mathcal{W}f|^2=1$ a.e.; this implies $f=1$ (via Cramer's theorem). Furthermore, $f_k$ is normalized in $L^2(dm)$, therefore along a further subsequence, $f_k \rightharpoonup g$ weakly in $L^2(dm)$ and by uniqueness of weak limits, $g=1$. This yields $f_k \rightarrow 1$ in $L^2(dm)$, a contradiction, therefore the result follows.     
\end{proof}

\begin{proof}[Proof of Theorem \ref{aat}]
Let $f \in L^2(dm)$ be normalized and $T$ the optimal transport map between $|U^*|^2dx$ and $|\widehat{U^*f}|^2dx$. It follows that $T=\nabla \phi$ satisfies 
\begin{align*}
\log \det D^2 \phi &= \log \frac{|U^*|^2}{|\widehat{U^*f}(T)|^2}\\
&=\frac{n}{2}\log2-2\pi|x|^2-\log |\widehat{U^*f}(T)|^2
\end{align*}
so that 
$$
\int |U^*|^2\log |\widehat{U^*f}(T)|^2 dx+\int |U^*|^2\log \det D^2 \phi dx = \frac{n}{2} \log(2) - 2\pi \int |x|^2|U^*|^2dx. 
$$ 
In particular, let $E_{dm}(|\mathcal{W}f|^2)=\int |\mathcal{W}f|^2 \log |\mathcal{W}f|^2 dm$, so that if $\psi(x)=\phi(x)-\frac{1}{2}|x|^2$ and $\lambda_i$ are the eigenvalues of $D^2 \psi$, then   
\begin{align*}
2\pi \int |x|^2 |\widehat{U^*f}|^2dx-2\pi \int |x|^2|U^*|^2dx&=\int \log D^2 \phi dm+E_{dm}(|\mathcal{W}f|^2)\\
& \le \int \sum_{i=1}^n \log(1+\lambda_i) dm+E_{dm}(|\mathcal{W}f|^2)\\
& \le \int \Delta \psi dm+E_{dm}(|\mathcal{W}f|^2)\\
& \le 4\pi \int (T(x)-x) \cdot x dm+E_{dm}(|\mathcal{W}f|^2)\\
& \le 4 \pi W_2(|U^*|^2dx, |\widehat{U^*f}|^2dx) m_2(dm)+E_{dm}(|\mathcal{W}f|^2). 
\end{align*}
Next, 
\begin{align*}
4\pi^2 \int |x|^2 |\widehat{U^*f}|^2 dx&= \int |\widehat{\nabla U^* f}|^2dx\\
&=\int |\nabla U^* f|^2dx\\
&=4\pi^2 \int |x|^2|U^*f|^2 dx- 4 \pi \int x \cdot \nabla f f dm +\int |\nabla f|^2 dm   
\end{align*}
and since 
$$
-2 \int x \cdot \nabla f f dm=n-4\pi\int|x|^2|f|^2dm
$$
it follows that
\begin{align*} 
2\pi \int |x|^2dm-2\pi \int |x|^2|f|^2dm +\frac{1}{2\pi} \int |\nabla f|^2 dm &\le 2\sqrt{\pi n} W_2(dm, |\mathcal{W}f|^2dm) +E_{dm}(|\mathcal{W}f|^2)\\
&\le 2\sqrt{\pi n} W_2(dm, |\mathcal{W}f|^2dm) +\delta_c(f)\\
&\le 2\sqrt{\pi n} \sqrt{\delta_c(f)}+\delta_c(f).
\end{align*}
\end{proof}

\begin{proof}[Proof of Corollary \ref{c3}]
Suppose $\delta_c(f_k) \rightarrow 0$ as $k\rightarrow \infty$, but 
$$
\limsup_{k \rightarrow \infty} \int |\nabla f_k|^2 dm>0.
$$
Then since $f_k\in \{f:\int |f|^4dm \le C\}$, it follows that along a subsequence, $f_k^2 \rightharpoonup f^2$ weakly in $L^2(dm)$ for some $f \in L^4(dm)$. In particular, 
$$
\int |x|^2|f_k|^2 dm \rightarrow \int |x|^2|f|^2 dm 
$$
as $k \rightarrow \infty$ and Theorem \ref{mainthm1} implies $f=1$. Theorem \ref{aat} therefore yields 
$$
\int |\nabla f_k|^2 dm \rightarrow 0 
$$
as $k \rightarrow \infty$, a contradiction. 
\end{proof}

\begin{proof}[Proof of Corollary \ref{c4}]
Suppose $\delta_c(f_k) \rightarrow 0$ and 
$$\limsup_{k \rightarrow \infty} \|f_k-1\|_{L^1(d\gamma)} >0.$$
Theorem \ref{aat} implies 
$$
\sup_k \int |\nabla f_k|^2 dm \le M 
$$
and up to a subsequence, it follows that $f_k \rightarrow f$ in $L^2(dm)$ as $k \rightarrow \infty$, for some $f \in L^2(dm)$ and Theorem \ref{mainthm1} implies $f=1$, a contradiction.  
\end{proof}

\begin{proof}[Proof of Theorem \ref{mainthm2}]
Let $\{f_j\}$ be any subsequence of the original sequence. We will show that there exists a further subsequence $\{f_{j(k)}\}$ that converges to 1 in $L^{1}(\R^{n},d\gamma)$.
By \eqref{ineq:L1_L2}, it suffices to show that $\sqrt{f_{j(k)}}\to 1$ in $L^{2}(\R^{n}, d\gamma)$. Since the deficit converges to zero, $\{\I(f_{j})\}_{j\geq1}$ is uniformly bounded in $j$. Indeed, it follows from the HWI inequality \eqref{ineq:HWI} and Young's inequality that
\begin{eqnarray*}
	\H(f)
	&\leq& W_{2}(fd\gamma,d\gamma)\sqrt{\I(f)}-\frac{1}{2}W_{2}^{2}(fd\gamma,d\gamma) \\
	&\leq& \frac{1}{2t}\I(f)+\frac{t-1}{2}W_{2}^{2}(\gamma,\nu)
\end{eqnarray*}
for $t>1$. Since $\H(f)=\frac{1}{2}\I(f)-\delta(f)$, we have
\begin{eqnarray*}
	\frac{t-1}{2t}\I(f)
	\leq \delta(f)+\frac{t-1}{2}W_{2}^{2}(fd\gamma,d\gamma)
\end{eqnarray*}
If $T:\R^{n}\to\R^{n}$ transports $\gamma$ onto $\nu$, then
\begin{eqnarray*}
	W_{2}^{2}(fd\gamma,d\gamma)
	&=& \int_{\R^{n}}|T(x)-x|^{2}d\gamma \\
	&\leq& 2\Big(\int_{\R^{n}}|x|^{2}fd\gamma+\int_{\R^{n}}|x|^{2}d\gamma\Big) \\
	&\leq& 2(n+M),
\end{eqnarray*}
which shows that $\{\I(f_{j})\}_{j\geq1}$ is uniformly bounded in $j$.
Let $h_{j}=f_{j}\gamma$ where $\gamma(x)=(2\pi)^{-\frac{n}{2}}e^{-\frac{|x|^{2}}{2}}$, then
\begin{eqnarray*}
	\I(f_{j})
	&=& 4\int_{\R^{n}}|\nabla(\sqrt{f_{j}})|^{2}d\gamma \\
	&=& 4\int_{\R^{n}}|\nabla(\sqrt{h_{j}})-\sqrt{f_{j}}\nabla(\sqrt{\gamma})|^{2}dx \\
	&=& 4\int_{\R^{n}}|\nabla(\sqrt{h_{j}})|^{2}dx-2n+\int_{\R^{n}}|x|^{2}d\nu_{j}.
\end{eqnarray*}
So $\{\sqrt{h_{j}}\}_{j\geq1}$ is bounded in $W^{1,2}(\R^{n})$.

Let $\Omega\subset\R^{n}$ be a bounded domain. The Rellich-Kondrashov theorem says that there exists a subsequence $\{j(k)\}_{k\geq 1}$ such that $\sqrt{h_{j(k)}}$ converges to a function $g$ in $L^{2}(\Omega)$. Since $\sqrt{h_{j}}$ is nonnegative for all $j$, we let $g=\sqrt{f\gamma}$. 

We claim that $f=1$ a.e. in $\Omega$. Let $d\nu_{j}=f_{j}d\gamma$. Since $\delta(f_{j(k)})$ converges to 0, we have $W_{1}(\nu_{j(k)},\gamma)\to 0 $ by Theorem \ref{mainthm1}. This implies that $\nu_{j(k)} \rightharpoonup \gamma$ weakly, that is,
\begin{eqnarray*}
	\lim_{k\to\infty}\int_{\R^{n}}\varphi d\nu_{j(k)}=\int_{\R^{n}}\varphi d\gamma
\end{eqnarray*}
for all $\varphi\in C^{0}_{b}(\R^{n})$. Let $\varep>0$ and $\varphi$ be a bounded continuous function such that $|\varphi|\leq K$ for some $K>0$. We pick $N\in\N$ such that
\begin{eqnarray*}
	\Abs{\int_{\Omega}\varphi d\nu_{j(k)}-\int_{\Omega}\varphi d\gamma}\leq \frac{\varep}{2}, \quad
	\Abs{\int_{\Omega}|\sqrt{f_{j(k)}}-\sqrt{f}|^{2}d\gamma}\leq \frac{\varep^{2}}{16K^{2}}
\end{eqnarray*}
for any $k\geq N$. Since $\int_{\Omega}f_{j(k)}d\gamma\leq 1$ for all $k$ and 
\begin{eqnarray*}
	\int_{\Omega}fd\gamma
	\leq \int_{\Omega}|\sqrt{f_{j(k)}}-\sqrt{f}|^{2}d\gamma + \int_{\Omega}f_{j(k)}d\gamma,
\end{eqnarray*}
we obtain $\int_{\Omega}fd\gamma\leq 1$. One can see that
\begin{eqnarray*}
	\Abs{\int_{\Omega}(f-1)\varphi d\gamma}
	&\leq & \Abs{\int_{\Omega}(f_{j(k)}-1)\varphi d\gamma} + \Abs{\int_{\Omega}(f-f_{j(k)})\varphi d\gamma} \\
	&\leq & \frac{\varep}{2} + K\int_{\Omega}|f-f_{j(k)}|d\gamma\\
	&\leq & \frac{\varep}{2} + K\Paren{\int_{\Omega}|\sqrt{f}-\sqrt{f_{j(k)}}|^{2}d\gamma}^{\frac{1}{2}}\Paren{\int_{\Omega}|\sqrt{f}+\sqrt{f_{j(k)}}|^{2}d\gamma}^{\frac{1}{2}}\\
	&\leq & \varep. 
\end{eqnarray*}
This holds for all $\varep>0$ and all $\varphi\in C^{0}_{b}(\Omega)$. Thus, we conclude that $f=1$ a.e. in $\Omega$.

We will now extend this claim to $\R^{n}$. The diagonalization argument enables us to do it. Let $B_{k}:=\{x\in\R^{n}:|x|<k\}$ for each $k\in\N$. First, choose a subsequence $\{f_{1,j}\}_{j\geq 1}$ such that $\sqrt{f_{1,j}}\to 1$ in $L^{2}(B_{1},d\gamma)$ as $j\to\infty$. On $B_{2}$, we can find a further subsequence $\{f_{2,j}\}_{j\geq 1}\subseteq\{f_{1,j}\}_{j\geq 1}$ such that $\sqrt{f_{2,j}}\to 1$ in $L^{2}(B_{2},d\gamma)$ as $j\to\infty$. Iterating this procedure, we have $\{f_{k,j}\}_{j,k\geq 1}$ such that $\sqrt{f_{k,j}}\to 1$ in $L^{2}(B_{k},d\gamma)$ as $j\to \infty$. Define $f^{(k)}:=f_{k,k}$ and let $\varep>0$.  

Since $\nu_{j}$ converges weakly to $\gamma$, the family $\{\nu_{j}\}$ is tight by Prokhorov's theorem. Thus, we can choose $N_{1}\in\N$ be such that $\int_{\R^{n}\setminus B_{k}}d\nu_{j}<\frac{\varep}{8}$ and $\int_{\R^{n}\setminus B_{k}}d\gamma<\frac{\varep}{8}$ for all $k\geq N_{1}$. By definition, there exists $N_{2}\in\N$ such that $\int_{B_{k}}|\sqrt{f^{(k)}}-1|^{2}d\gamma<\frac{\varep}{2}$ for all $k\geq N_{2}$. Combining our observation, we have
	\begin{eqnarray*}
		\Abs{\int_{\R^{n}}|\sqrt{f^{(k)}}-1|^{2}d\gamma}
		&\leq & \Abs{\int_{B_{k}}|\sqrt{f^{(k)}}-1|^{2}d\gamma}+2\Abs{\int_{\R^{n}\setminus B_{k}}(f^{(k)}+1)d\gamma}\\
		&<&\varep
	\end{eqnarray*}
	for any $k\geq \max\{N_{1},N_{2}\}$. Therefore, we conclude that $\sqrt{f^{(k)}}\to 1$ in $L^{2}(\R^{n}, d\gamma)$ as desired.
\end{proof}

\appendix
\section{A Stability result in terms of Kolmogorov distance}

In this section, we prove stability results of LSI in terms of the Kolmogorov distance. We assume that the probability measure belongs to $\cP_{2}^{M}$ and also satisfies further integrability and second moment assumptions. For probability measures $\mu$ and $\nu$ on $\R$, the Kolmogorov distance is given by
\begin{eqnarray*}
	\dkol{\mu,\nu}=\sup_{x\in\R}|\mu((-\infty,x])-\nu((-\infty,x])|.
\end{eqnarray*}

\begin{theorem}\label{mainthm4}
	Let $f$ be a symmetric nonnegative function on $\R$ 
	and $d\mu=fd\gamma\in \cP_{2}^{M}(\R)$ with $m_{2}(\mu)=k$. Let $v(x)=f(\frac{x}{\sqrt{2}})^{2}$ and assume that $d\nu:=vd\gamma$ is a probability measure. 
	Then there exists $\varep_{0}>0$ such that if $\delta(v)\le\varep\leq\varep_{0}$, then
	\begin{eqnarray}\label{eq:thm4_1}
		\dkol{\mu,\gamma_{\varep}}\le
		\frac{C_{k}}{\sqrt{\log\frac{1}{\varep}}}
	\end{eqnarray}
	where $C_{k}$ depends on $k$ and $\gamma_{\varep}$ is a Gaussian measure given by
	\begin{eqnarray*}
		d\gamma_{\varep}
		= \frac{1}{\sqrt{4\pi\sigma_{\varep}^{2}}}e^{-\frac{|x|^{2}}{4\sigma_{\varep}^{2}}}dx,
	\end{eqnarray*}
	for some $\sigma_{\varep}^{2}>0$ depending on $\varep$.
\end{theorem}

\begin{theorem}\label{mainthm5}
	Let $f$ be a symmetric nonnegative function on $\R$, $d\mu=fd\gamma$, and $m_{2}(\mu)=1$. Let $v(x)=f(\frac{x}{\sqrt{2}})^{2}$ and assume that $d\nu:=vd\gamma$ is a probability measure. Let $\Psi(t)=e^{-\frac{1}{t^{2}}}$, then there exist constants $c_{1}$ and $c_{2}$ such that
	\begin{eqnarray*}
		\delta(v)
		\geq c_{1}\Psi(c_{2}\dkol{\mu,\gamma}).
	\end{eqnarray*}
\end{theorem}

\begin{remark}
Note that for $d\gamma_{\frac{1}{\sqrt{2}}}=\frac{1}{\sqrt{\pi}}e^{-|x|^{2}}dx$,
\begin{eqnarray*}
	\delta(v)
	&=& \int |\nabla f|^{2}d\gamma_{\frac{1}{\sqrt{2}}}-\int |f|^{2}\log |f|^{2}d\gamma_{\frac{1}{\sqrt{2}}}.
\end{eqnarray*}
\end{remark}
\begin{remark}
By Proposition \ref{probmetricbound_general}, Theorem \ref{mainthm1} implies that
\begin{eqnarray*}
	\delta(f)\geq C_{M}\min\{\dkol{\mu,\gamma}^{2},\dkol{\mu,\gamma}^{8}\}.
\end{eqnarray*}
On the other hand, it follows from Proposition \ref{probmetricbound_W_1} that Theorem \ref{mainthm5} implies
\begin{eqnarray*}
	\delta(v)
		\geq c_{1}\Psi(c_{2}W_{1}(\mu,\gamma)^{2}).
\end{eqnarray*}
Note that if $t$ is small then $\Psi(t)$ is bounded by the map $t\mapsto \min\{t^{2},t^{8}\}$. So the stability result of Theorem \ref{mainthm1} is stronger than that of Theorem \ref{mainthm5}. Notice also that these Theorem \ref{mainthm5} has a scaled version of the deficit $\delta(v)$.
\end{remark}

The key ingredients to the proofs are the LSI stability result of \cite{Feo2017} and the stability of Cram\'er's theorem. To state these results, we recall some notations. Let $g(x)=2^{\frac{n}{4}}e^{-\pi|x|^{2}}$ and $dm=g^{2}(x)dx$. If $u_{f}(x)=f(2\sqrt{\pi}x)^{\frac{1}{2}}$, then we have $\delta_{c}(u_{f})=\delta(f)$ by scaling. 

\begin{theorem}[{\cite[Theorem 4.1]{Feo2017}}]\label{FFthm}
	Let $f\in L^{2}(dm)$, $f(x)=f(-x)$, $\norm{f}_{2}=1$, and $h=fg$. Then there exists a constant $C>0$ such that 
	\begin{eqnarray*}
		\int_{\R} |h\ast h-g\ast g|^{2}dx
		\leq C\delta_{c}(f)^{\frac{1}{4}}
		(\norm{h-g}_{\frac{6}{5}}^{2}+\norm{h-g}_{2})^{\frac{3}{2}}	.
	\end{eqnarray*}
\end{theorem}

Cram\'er's theorem says that if the sum of two independent random variables has a normal distribution, then both random variables are normal. We recall the stability of Cram\'er's theorem. Let $X$ and $Y$ be independent random variables with distribution functions $F$ and $G$ respectively, then Kolmogorov distance between $X$ and $Y$ is
\begin{eqnarray*}
	\dkol{F,G}=\sup_{x\in\R}|F(x)-G(x)|.
\end{eqnarray*}
Let $F\ast G$ be the distribution of the sum $X+Y$ so that it is defined by
\begin{eqnarray*}
	F\ast G(x):=\int_{\R}F(x-y)d	G(y).
\end{eqnarray*} 
If $p_{1}$ and $p_{2}$ are density functions of $X$ and $Y$, one can write it as
\begin{eqnarray*}
	F\ast G(x)=\int_{-\infty}^{x}p_{1}\ast p_{2}(t)dt.
\end{eqnarray*}
We use $\gamma_{b, \sigma}(x)=\frac{1}{\sqrt{2\pi\sigma^{2}}}e^{-\frac{|x-b|^{2}}{2\sigma^{2}}}$ and $\Phi_{b,\sigma}$ to denote the centered Gaussian density with variance $\sigma^{2}$ and its distribution function. We also denote by $\Phi_{\sigma}:=\Phi_{0,\sigma}$ and $\Phi:=\Phi_{0,1}$. The following stability result was proven in \cite{Golinskiui1991} and \cite{Senatov1976}.

\begin{theorem}[{\cite[Theorem 2.2]{Bobkov2016}}]\label{Bobkov}
	Let $\varep>0$ and $N=N(\varep)=1+\sqrt{2\log(1/\varep)}$. Let $X_{1}, X_{2}$ be random variable with distribution functions $F_{1}, F_{2}$. We also put
	\begin{eqnarray*}
		a_{i}=\int_{-N}^{N}xdF_{i}(x),\quad
		\sigma_{i}^{2}=\int_{-N}^{N}x^{2}dF_{i}(x)-a_{i}^{2}
	\end{eqnarray*}
	for $i=1,2$. Suppose that $F_{1}$ and $F_{2}$ have median zero and $\sigma_{1},\sigma_{2}>0$.
	If $\dkol{F_{1}\ast F_{2},\Phi}\leq \varep<1$, then there exist absolute constants $C_{1}$ and $C_{2}$ such that for $i=1,2$, 
	\begin{eqnarray*}
		\dkol{F_{i},\Phi_{a_{i},\sigma_{i}}}
		\leq \frac{C_{i}}{\sigma_{i}\sqrt{\log(1/\varep)}}\min\{\frac{1}{\sqrt{\sigma_{i}}},\log\log\frac{e^{e}}{\varep}\}.
	\end{eqnarray*}
\end{theorem}
This result is more or less restrictive in a sense that the medians are zero and $\sigma_{i}$, $i=1,2$ depend on $\varep$. A more general version of the stability result was obtained in \cite{Bobkov2016}. 

\begin{theorem}[{\cite[Theorem 2.3]{Bobkov2016}}]\label{Bobkov2}
	Let $X_{1}$, $X_{2}$ be independent random variables with $\E[X_{1}]=\E[X_{2}]=0$ and $\Var[X_{1}+X_{2}]=1$. For $i=1,2$,  let $F_{i}$ be the distribution function for $X_{i}$ and $v_{i}^{2}=\Var(X_{i})$. If  $\dkol{F_{1}\ast F_{2},\Phi_{1}}\leq\varep<1$, then there exists a constant $C$ such that
	\begin{eqnarray*}
		\dkol{F_{i},\Phi_{v_{i}}}\leq \frac{C}{v_{i}\sqrt{\log\frac{1}{\varep}}}\min\{\frac{1}{\sqrt{v_{i}}},\log\log\frac{e^{e}}{\varep}\}
	\end{eqnarray*}
	for $i=1,2$.
\end{theorem}

The following lemma gives an interpolation estimate under a second moment assumption, which allows to connect the stability of Cram\'er's theorem with that of the LSI. 
\begin{lemma}\label{L1L2lemma}
	Let $u$ be a nonnegative function in $L^{1}(dx)\cap L^{2}(dx)$ such that 
	\begin{eqnarray*}
		\int_{\R}x^{2}u(x)dx=k\norm{u}_{1}<\infty
	\end{eqnarray*}
	for some $k>0$. Then we have $\norm{u}_{1}\leq e^{\frac{k+1}{2}}\norm{u}_{2}$.
\end{lemma}
\begin{proof}
	Let $p(x)=u(x)/\norm{u}_{1}$ and $q(x)=\frac{1}{\sqrt{\pi}}e^{-x^{2}}$. Since $\varphi=x\log x$ for $x\geq 0$ is convex ($\varphi(0)=0$), one can see by Jensen's inequality that
	\begin{eqnarray*}
		\int_{\R}p(x)\log\frac{p(x)}{q(x)}dx
		&=& \int_{\R}\varphi(\frac{p(x)}{q(x)})q(x)dx \\
		&\geq & \varphi(\int_{\R}p(x)dx) \\
		&=&0.
	\end{eqnarray*}
	So, we have
	\begin{eqnarray}\label{entropybound}
		\int_{\R}p(x)\log p(x)dx
		&\geq& \int_{\R}p(x)\log q(x)dx\\
		&=&-\int_{\R}x^{2}p(x)dx-\frac{1}{2}\log\pi\nonumber\\
		&\geq& -(k+1).\nonumber
	\end{eqnarray}
	Let $1\leq p_{0},p_{1}\leq 2$, $\theta\in(0,1)$ and $\frac{1}{p_{\theta}}=\frac{1-\theta}{p_{0}}+\frac{\theta}{p_{1}}$. It follows from H\"older inequality that
	\begin{eqnarray}\label{logconvex}
		\norm{u}_{p_{\theta}}\leq \norm{u}_{p_{0}}^{1-\theta}\norm{u}_{p_{1}}^{\theta}.
	\end{eqnarray}
	This implies that the map $p\mapsto J(p):=\log\norm{f}_{p}^{p}$ is convex on $[1,2]$. On the other hand, the derivative of $J(p)$ is given by
	\begin{eqnarray*}
		\frac{d}{dp}J(p)
		=\frac{1}{\norm{u}_{p}^{p}}\int_{\R}|u|^{p}\log|u|dx.
	\end{eqnarray*}
	By the convexity of $J(p)$, we have $J(2)-J(1)\geq J'(1)$. So, we apply \eqref{entropybound} to obtain
	\begin{eqnarray*}
		\log\norm{u}_{2}^{2}-\log\norm{u}_{1}
		&\geq& \frac{1}{\norm{u}_{1}}\int_{\R}|u|\log|u|dx \\
		&=& \int_{\R}p(x)\log p(x)dx+\log\norm{u}_{1}\\
		&>& -(k+1)+\log\norm{u}_{1},
	\end{eqnarray*}
	which yields the desired result.
\end{proof}

\begin{proof}[Proof of Theorem \ref{mainthm4} ]
	Let $h(x)=\wt{f}(x)g(x)$ and $\wt{f}(x)=f(\sqrt{2\pi}x)$ then one can easily see that
	\begin{eqnarray}\label{hnormalization}
		\int_{\R}|h|^{2}dx=1, \quad
		\int_{\R}hdx=2^{\frac{1}{4}}, \quad
		\int_{\R}x^{2}hdx=\frac{2^{-\frac{3}{4}}k}{\pi}.
	\end{eqnarray}
	Let $X_{1}, X_{2}$ be i.i.d. random variables with density $p(x)=\frac{2^{-1/4}}{\sqrt{\pi}}h(\frac{x}{\sqrt{\pi}})$ and distribution function $F$. Note that $F$ has median zero and $\Var[X_{1}]=\frac{k}{2}$. 
	Since Kolmogorov distance is bounded by total variation, one can see that
	\begin{eqnarray*}
		\dkol{F\ast F,\Phi_{1}}
		\leq \frac{1}{2}\int_{\R}|p\ast p(x)-\gamma(x)|dx.
	\end{eqnarray*}
	Since we have $h\ast h(x) = \sqrt{2\pi} p\ast p(\sqrt{\pi} x)$ and $g\ast g(x)=\sqrt{2\pi}\gamma(\sqrt{\pi}x)$, we obtain
	\begin{eqnarray*}
		\dkol{F\ast F,\Phi_{1}}
		\leq \frac{1}{2\sqrt{2}}\int_{\R}|h\ast h(x)-g\ast g(x)|dx.
	\end{eqnarray*}
	Let $u:=h\ast h-g\ast g$, then we have $\norm{u}_{1}\leq 2\sqrt{2}$ and
	\begin{eqnarray*}
		\int_{\R}x^{2}|u|dx
		&\leq&  \int_{\R}x^{2}(h\ast h)(x)dx + \int_{\R}x^{2}(g\ast g)(x)dx \\
		&\leq & 2^{\frac{5}{4}}\Paren{\int_{\R}x^{2}h(x)dx + \int_{\R}x^{2}g(x)dx } \\
		&\leq &C(k+1).
	\end{eqnarray*}
	 By Lemma \ref{L1L2lemma}, we have $\norm{u}_{1}\leq C_{k}\norm{u}_{2}$ where $C_{k}>0$ depends only on $k$. Combining our observation and Theorem \ref{FFthm}, we obtain
	\begin{eqnarray*}
		\dkol{F\ast F,\Phi_{1}}
		\leq C_{k}
			(\norm{h-g}_{\frac{6}{5}}^{2}+\norm{h-g}_{2})^{\frac{3}{4}}\delta_{c}(\wt{f})^{\frac{1}{8}}
	\end{eqnarray*}
	where $\wt{f}=f(\sqrt{2\pi}x)$. 
	Note that $\delta_{c}(\wt{f})=\delta(v)$.
	It follows from  \eqref{logconvex} and \eqref{hnormalization} that $(\norm{h-g}_{\frac{6}{5}}^{2}+\norm{h-g}_{2})^{\frac{3}{4}}$ is bounded by a numerical constant and that
	\begin{eqnarray*}
		\dkol{F\ast F, \Phi_{1}}
		\leq C_{k}\delta(v)^{\frac{1}{8}}.
	\end{eqnarray*}
	
	Choose $\varep_{0}>0$ such that $C_{k}\varep_{0}^{\frac{1}{8}}< 1$. Let $\varep>0$ be such that $\delta(v)<\varep<\varep_{0}$, and put $\eta=C_{k}\varep^{\frac{1}{8}}$, $N=N(\eta)=1+\sqrt{2\log(1/\eta)}$ and
	\begin{eqnarray*}
		\sigma(\eta)^{2}=\int_{-N(\eta)}^{N(\eta)}x^{2}p(x)dx.
	\end{eqnarray*}
	Note that $\sigma(\eta)^{2}\nearrow \Var[X_{1}]=\frac{1}{2}m_{2}(\mu)$ as $\eta\to 0$. So, we choose $\varep_{0}$ small enough so that $\frac{1}{4}m_{2}(\mu)<\sigma(\eta)^{2}$ for all $\varep<\varep_{0}$. It then follows from Theorem \ref{Bobkov} that
	\begin{eqnarray*}
		\dkol{F,\Phi_{\sigma(\eta)}}
		&\leq &	\frac{C}{\sigma(\eta)\sqrt{\log(1/\eta)}}\min\{\frac{1}{\sqrt{\sigma(\eta)}},\log\log\frac{e^{e}}{\eta}\}\\
		&\leq & \frac{C}{\sigma(\eta)^{\frac{3}{2}}\sqrt{\log(1/\eta)}}\\
		&\leq &	\frac{C}{m_{2}(\mu)^{\frac{3}{4}}\sqrt{\frac{1}{8}\log(\frac{1}{\varep})-\log C_{k}}}\\
		&\leq & \frac{C_{k}}{\sqrt{\log\frac{1}{\varep}}}.
	\end{eqnarray*}
	By change of variables, we have $\dkol{F,\Phi_{\sigma(\eta)}}=\dkol{\mu,\gamma_{\varep}}$ where
	\begin{eqnarray*}
		d\gamma_{\varep}
		= \frac{1}{\sqrt{4\pi\sigma(\eta)^{2}}}e^{-\frac{|x|^{2}}{4\sigma(\eta)^{2}}}dx,
	\end{eqnarray*}
	which yields \eqref{eq:thm4_1}. 
\end{proof}
\begin{proof}[Proof of Theorem \ref{mainthm5} ]
	Let $h(x)=\wt{f}(x)g(x)$ and $\wt{f}(x)=f(\sqrt{2\pi}x)$.
	Let $X_{1}, X_{2}$ be i.i.d. random variables with density $p(x)=\frac{2^{-1/4}}{\sqrt{\pi}}h(\frac{x}{\sqrt{\pi}})$ and distribution function $F$.
	Since $m_{2}(\mu)=1$, we have $\Var(X_{1})=\Var(X_{2})=\frac{1}{2}$. The same argument then leads to
	\begin{eqnarray*}
		\dkol{F\ast F, \Phi_{1}}
		\leq c_{1}\delta_{c}(\wt{f})^{\frac{1}{8}}
		= c_{1}\delta(v)^{\frac{1}{8}}
	\end{eqnarray*}
	for some numerical constant $c_{1}$. So, we choose $\varep_{0}>0$ such that $c_{1}\varep_{0}^{\frac{1}{8}}<1$. Assume $\delta(v)<\varep<\varep_{0}$. We apply Theorem \ref{Bobkov2} to obtain
	\begin{eqnarray*}
		\dkol{F,\Phi_{\frac{1}{\sqrt{2}}}}
		<\frac{c_{2}}{\sqrt{\log\frac{1}{\varep}}}.
	\end{eqnarray*}
	Note that $\dkol{F,\Phi_{\frac{1}{\sqrt{2}}}}=\dkol{\mu,\gamma}$ by change of variables. Since $\Psi(s)$ is the inverse of the map $t\mapsto \frac{1}{\sqrt{\log\frac{1}{t}}}$, we have $\delta(v)\geq \Psi(c_{3}\dkol{\mu,\gamma})$ for some constant $c_{3}$. Since the map $\Psi$ is bounded, there exists a constant $c_{4}$ such that $\delta(v)\geq c_{4}\Psi(c_{3}\dkol{\mu,\gamma})$ as desired.
\end{proof}

\section{Probability metrics}
This section is devoted to introduce probability metrics and provide their relations. The following is  based on \cite{Bobkov2016,Gibbs2002,Villani2003}.

\begin{definition}
	Let $\Omega$ be a measurable space.
	Let $\lambda$ be a measure on $\Omega$. For probability measures $d\mu=fd\lambda$ and $d\nu=gd\lambda$, the Hellinger distance is defined by 
	\begin{eqnarray*}
		\dhel(\mu,\nu)=\Big(\int_{\Omega}|\sqrt{f}-\sqrt{g}|^{2}d\lambda\Big)^{\frac{1}{2}}.
	\end{eqnarray*}	
\end{definition}
Note that $\dhel$ is a metric and $0\leq \dhel(\mu,\nu)\leq \sqrt{2}$. 

\begin{definition}
	Let $\Omega$ be a measurable space.
	Let $\mu$ and $\nu$ be probability measures on $\Omega$.
	The total variation distance is 
	\begin{eqnarray*}
		\dtv{\mu,\nu}
		=\sup_{h}\Abs{\int_{\Omega}hd\mu-\int_{\Omega}hd\nu}
	\end{eqnarray*}	
	where the supremum is taken over all measurable functions $h:\Omega\to\R$ with $|h(x)|\leq 1$.
\end{definition}

\begin{definition}
	Let $(\Omega,d)$ be a Polish space. Let $\mu$ and $\nu$ be probability measures on $\Omega$. For a Borel set $B$ and $\varep>0$, $B^{\varep}=\{x\in\Omega:\inf_{y\in B}d(x,y)\leq \varep\}$.
	The Prokhorov metric is defined by
	\begin{eqnarray*}
		\dpro(\mu,\nu)
		=\inf\{\varep>0: \mu(B)\leq \nu(B^{\varep})+\varep 
		\text{ for all Borel sets } B\}.
	\end{eqnarray*}
\end{definition}
If $X$ and $Y$ random variables with the laws $\mu$ and $\nu$, then it follows from Strassen's theorem that
\begin{eqnarray*}
	\dpro(\mu,\nu)
	=\inf_{\P}\{\varep>0: \P(d(X,Y)>\varep)<\varep\}
\end{eqnarray*}
where the infimum is taken over all joint probability $\P$ of $X$ and $Y$. Similarly, we have
\begin{eqnarray*}
	\dtv{\mu,\nu}
	=\inf\bE[1_{X\neq Y}]
	=\sup\{\mu(F)-\nu(F):F\text{ closed}\}.
\end{eqnarray*}
\begin{definition}
	The Kolmogorov distance between two probability measures $\mu$ and $\nu$ on $\R$ is given by
	\begin{eqnarray*}
		\dkol{\mu,\nu}=\sup_{x\in\R}|\mu((-\infty,x])-\nu((-\infty,x])|.
	\end{eqnarray*}	
\end{definition}
If $F$ and $G$ are distribution functions of $\mu$ and $\nu$, then we denote by $\dkol{F,G}=\dkol{\mu,\nu}$. One can see that $0\leq \dkol{\mu,\nu}\leq 1$.

\begin{definition}
	Let $\mu$ and $\nu$ be probability measures on $\R$ with distribution functions $F$ and $G$.
	The L\`evy metric is defined by
	\begin{eqnarray*}
		\dlev(\mu,\nu)=\dlev(F,G)
		=\inf\{\varep>0: G(x-\varep)-\varep\leq F(x)\leq G(x+\varep)+\varep, \forall x\in\R\}.
	\end{eqnarray*}
\end{definition}

\begin{proposition}\label{probmetricbound_general}
	Let $\mu$ and $\nu$ be probability measures on $\R$, then we have
	\begin{eqnarray*}
		\dlev(\mu,\nu)
		\leq \min\{\dkol{\mu,\nu}, \dpro(\mu,\nu)\}
		\leq \max\{\dkol{\mu,\nu}, \dpro(\mu,\nu)\}
		\leq \min\{\dtv{\mu,\nu}, \sqrt{W_{1}(\mu,\nu)}\}.
	\end{eqnarray*}
\end{proposition}
\begin{proposition}
	Let $\mu$ be a probability measure on $\R$ and $\gamma$ the standard Gaussian measure on $\R$, then 
	\begin{eqnarray*}
		\dkol{\mu,\gamma}\leq  2\dpro(\mu,\gamma).
	\end{eqnarray*}
\end{proposition}

\begin{proposition}[{\cite[Proposition A.1.2]{Bobkov2016}}]\label{probmetricbound_W_1}
	Let	$\mu,\nu \in\cP_{2}^{M}(\R)$, then
	\begin{eqnarray*}
		W_{1}(\mu,\nu) 
		&\leq& 2\dlev(\mu,\nu)+2\sqrt{M}\dlev(\mu,\nu)^{1/2}, \\
		W_{1}(\mu,\nu)
		&\leq& 4\sqrt{M}\dkol{\mu,\nu}^{1/2}.
	\end{eqnarray*}
\end{proposition}

\begin{proposition}
	Let $\Omega$ be a measurable space.
	Let $\mu$ and $\nu$ be probability measures on $\Omega$, then
	\begin{eqnarray*}
		\dhel(\mu,\nu)^{2}\leq \dtv{\mu,\nu} \leq 2\dhel(\mu,\nu).
	\end{eqnarray*}
\end{proposition}

\bibliography{stablsi}
\bibliographystyle{amsplain}

\end{document}